\documentclass{amsart}
\usepackage{amsmath}
\usepackage{amsfonts}
\usepackage{amssymb}
\usepackage{amsthm}
\usepackage{enumerate}
\usepackage{mathtools}
\usepackage{footnote}

\newtheorem{theorem}{Theorem}

\newtheorem{remark}[theorem]{Remark}
\newtheorem{corollary}[theorem]{Corollary}

\newtheorem{conjecture}{Conjecture}
\newtheorem{example}[theorem]{Example}

\newcommand{\refT}[1]{Theorem~\ref{#1}}
\newcommand{\refC}[1]{Corollary~\ref{#1}}

\newcommand{\refR}[1]{Remark~\ref{#1}}

\newcommand{\refCon}[1]{Conjecture~\ref{#1}}
\newcommand{\refEx}[1]{Example~\ref{#1}}

\newcommand{\cF}{\mathcal{F}}
\newcommand{\cP}{\mathbb{P}}
\newcommand{\cE}{\mathbb{E}}

\title{Short rainbow cycles for families of small edge sets}
\author{He Guo}\thanks{Department of Mathematics and Mathematical Statistics, Ume\r{a} University, Ume\r{a} 90187, Sweden.  E-mail: {\tt he.guo@umu.se}. Research supported by the Kempe Foundation grant JCSMK23-0055.}
\date{June 2025}

\begin{document}
\begin{abstract}
   In 2019, Aharoni proposed a conjecture generalizing the Caceetta-H\"aggkvist conjecture: if an $n$-vertex graph~$G$ admits an edge coloring (not necessarily proper) with~$n$ colors such that each color class has size at least~$r$, then~$G$ contains a rainbow cycle of length at most $\lceil n/r\rceil$. Recent works~\cite{AG2023,ABCGZ2023,G2025} have shown that if a constant fraction of the color classes are non-star, then the rainbow girth is $O(\log n)$. In this note, we extend these results, and we show that even a small fraction of non-star color classes suffices to ensure logarithmic rainbow girth. We also prove that the logarithmic bound is of the right order of magnitude. Moreover, we determine the threshold fraction between the types of color classes at which the rainbow girth transitions from linear to logarithmic.
\end{abstract}
\maketitle

\section{Introduction}
The following is a well-known conjecture in graph theory by Caccetta and H\"{a}ggkvist~\cite{CaccettaHaggkvist} from 1978.

\begin{conjecture}[\cite{CaccettaHaggkvist}]\label{conj:CH}
    In an $n$-vertex digraph with minimum out-degree~$r$, there exists a directed cycle of length at most $\lceil \frac{n}{r}\rceil$.
\end{conjecture}
The conjecture is still open. For a history of the studying of this problem, see, e.g., the references in~\cite{SHEN2000167, shen2002caccetta, sullivan2006summary}.

Given a graph~$G$ and an edge coloring (not necessarily proper) with~$m$ colors $\lambda:E(G)\rightarrow [m]$, for each $i\in [m]$ the edge set~$\lambda^{-1}(i)$ is called a \emph{color class}, and a subgraph~$H$ of~$G$ is \emph{rainbow} if no two edges of~$H$ belong to the same color class. The \emph{rainbow girth} of~$G$ is the minimum length of a rainbow cycle in~$G$, which is defined to be~$\infty$ if there is no rainbow cycle.
In~\cite{aharoniconjecture}, a generalization of the Caccetta-H\"{a}ggkvist conjecture is raised by Aharoni.
\begin{conjecture}[\cite{aharoniconjecture}]\label{conj:A}
    For an $n$-vertex graph~$G$ and an edge coloring of~$G$ with~$n$ colors, if each color class is of size at least~$r$, then the rainbow girth at most $\lceil\frac{n}{r}\rceil$.
\end{conjecture}
See~\cite{AG2023} for a proof why~\refCon{conj:A} implies~\refCon{conj:CH}. And if~\refCon{conj:A} is true, then the bound is tight: Consider the $n$-vertex graph on~$\{x_1,\dots, x_n\}$ and each color class has the form $\{x_ix_{i+1},\dots, x_ix_{i+r}\}$ for $1\le i\le n$, where the indices are taken modulo~$n$.

Regarding~\refCon{conj:A}, DeVos et al.~\cite{DeVos} proved that the conjecture is true when~$r=2$. For the case that each color class is of size one or two, a more explicit bound on the rainbow girth is given in~\cite{ABCGZ2023}. 
\begin{theorem}[\cite{ABCGZ2023}]
    For an $n$-vertex graph and an edge coloring with~$n$ colors, if~$p$ color classes are of size 1 and~$n-p$ color classes are of size 2, then the rainbow girth is at most $\lceil \frac{n+p}{2}\rceil$.
\end{theorem}
The~$r=3$ case of~\refCon{conj:A} is still open, while Clinch et al.~\cite{Clinch2024} proved that for $r=3$ the rainbow girth is at most~$\frac{4n}{9}+7$. For general~$r$, Hompe and Spirkl~\cite{HS} proved that there exists a constant~$0<C\le 10^{11}$ such that the rainbow girth is at most~$C\frac{n}{r}$. Recently, Hompe and Huynh~\cite{HOMPE202580} proved a strengthening result that the rainbow girth is at most $\frac{n}{r}+\alpha_r$ for some positive constant~$\alpha_r$ depending on~$r$.

As the bound on the rainbow girth in~\refCon{conj:A} is linear in~$n$ and in all the known extremal examples the color classes are stars, another direction is to study when the rainbow girth is small, say, logarithmic in~$n$. 
\begin{theorem}[\cite{AG2023}]
There exists~$C>0$ such that for any $n$-vertex graph and edge coloring with~$n$ colors, if each color class contains a matching of size two, then the rainbow girth is at most $C\log n$.
\end{theorem}
In~\cite{G2025}, the result is strengthened to the following.
\begin{theorem}[\cite{G2025}]\label{thm:matchings}
    For any $\alpha > \frac{1}{2}$, there exists $C>0$ such that the following holds. Let~$G$ be an $n$-vertex graph admitting an edge coloring with~$n$ colors such that each color class is non-empty. If there are at least~$\alpha n$ color classes, each of which contains a matching of size two, then the rainbow girth of~$G$ is at most $C\log n$.
\end{theorem}
The example in~\cite[Section 3.1.2]{G2025}, which is an $n$-vertex graph satisfying all the other conditions in~\refT{thm:matchings} for $\alpha= \frac{1}{2}$ and having linear rainbow grith, shows that the condition $\alpha > \frac{1}{2}$ is tight to ensure logarithmic rainbow girth.

In this note, we further strengthen the above result.
\begin{theorem}\label{thm:main}
    For any $\alpha,\beta\ge 0$ with $2\alpha+\beta >1$, there exists $C>0$ such that the following holds. 
    Let~$G$ be an $n$-vertex graph admitting an edge coloring with~$n$ colors. If the family of color classes $\cF=(F_1,\dots,F_n)$ satisfies the following conditions:
    \begin{itemize}
\item        each color class is non-empty,
 \item    there exists $\cF_M\subseteq \cF$ such that $|\cF_M|= \alpha n$ and each color class in~$\cF_M$ contains a matching of size two,
 \item and there exists $\cF_S\subseteq \cF\setminus\cF_M$ such that $|\cF_S|=\beta n$ and each color class in~$\cF_S$ contains a star of size two.
    \end{itemize}
Then the rainbow girth of~$G$ is at most $C\log n$.
\end{theorem}
\begin{remark}\label{rmk:abtight}
 To have logarithmic rainbow girth, the condition $2\alpha+\beta >1$ in~\refT{thm:main} is tight. See~\refEx{ex:abtight}.   
\end{remark}
In~\refT{thm:mainstronger}, we prove a slightly stronger result, which allows the numbers of color classes to be less than those in~\refT{thm:main}. As an immediate implication, if each of the~$n$ color classes is of size at least two, then a small fraction of color classes that contain a matching of size two guarantee the rainbow girth to be logarithmic.
\begin{corollary}\label{cor:matching}
    For any $\alpha >0$, there exists $C>0$ such that the following holds. Let~$G$ be an $n$-vertex graph admitting an edge coloring with~$n$ colors such that each color class is of size at least two. If there are at least $\alpha n$ color classes, each of which contains a matching of size two, then the rainbow girth of~$G$ is at most $C\log n$.
\end{corollary}

Since a non-star edge set not containing a matching of size 2 is a triangle, in~\cite{ABCGZ2023} the following result regarding the case that each color class is a triangle is proved.
\begin{theorem}[\cite{ABCGZ2023}]\label{thm:triangleold}
    There exists $C>0$ such that for any $n$-vertex graph and edge coloring with~$n$ colors, if each color class contains a triangle, then the rainbow girth is at most~$C\log n$.
\end{theorem}

In~\cite{G2025},~\refT{thm:triangleold} is extended to the following.
\begin{theorem}[\cite{G2025}]\label{thm:trianglematching}
    There exits $C>0$ such that for any $n$-vertex graph and edge coloring with~$n$ colors, if each color class contains either a matching of size two or a triangle, then the rainbow girth is at most $C\log n$.
\end{theorem}

We further extend~\refC{cor:matching} and~\refT{thm:trianglematching} and show that a small fraction of non-star color classes make the rainbow grith logarithmic.
\begin{theorem}\label{thm:nonstar}
    For any $\alpha >0$, there exists~$C>0$ such that the following holds. Let~$G$ be an $n$-vertex graph admitting an edge coloring with~$n$ colors such that each color class is of size at least two. If there are at least~$\alpha n$ color classes, each of which contains either a matching of size two or a triangle, then the rainbow girth of~$G$ is at most $C \log n$.
\end{theorem}
\begin{remark}
    To have logarithmic rainbow girth, the condition $\alpha>0$ is tight. See~\refR{rmk:abtight} for the case $\alpha=0$ and $\beta =1$.
\end{remark}
\refT{thm:nonstarex} later shows that among~$n$ color classes, each of size at least two, the presence of~$\omega(\sqrt{n})$ non-star classes already forces the rainbow girth to be~$o(n)$.

Noting that any set of~$r$ edges is a subset of the edge set of the complete graph~$K_{2r}$, a special case of the following~\refT{thm:lb} (taking the uniformity~$k=2$ and setting~$\mathcal{F}=(C_\ell)_{\ell\ge 2}$ as the collection of cycles) shows that the logarithmic bounds in the above results are of the right order of magnitude.

To state~\refT{thm:lb} concerning hypergraphs,
a finite collection~$F$ of~$k$-sets is called a \emph{$k$-uniform hypergraph}, or a \emph{$k$-graph} for brevity, whose \emph{vertex set} is~$V(F):=\cup_{e\in F}e$. An element of~$F$ is called a \emph{(hyper)edge} of~$F$. A  \emph{complete $k$-graph on~$t$ vertices}, denoted by~$K_t^{(k)}$, is the collection of all $k$-subsets of the $t$-vertex set. The \emph{density} of a hypergraph~$F$ is $\frac{|F|}{|V(F)|}$. A \emph{$\frac{1}{k-1}$-dense sequence} of~$k$-graphs is a sequence $(F_\ell)_{\ell\ge 2}$ of~$k$-graphs such that for each~$\ell$, $|F_\ell|=\ell$ and the density of~$F_\ell$ is at least $\frac{1}{k-1}$, which is equivalent to~$|V(F_\ell)|\le\ell(k-1)$. For example, $(B_\ell^{(k)})_{\ell\ge 2}$, where~$B_\ell^{(k)}$ is a Berge $k$-cycle of length~$\ell$, is a $\frac{1}{k-1}$-dense sequence. 
For a hypergraph, the definition of color classes and rainbow subsets is just same as the graph case.


\begin{theorem}\label{thm:lb}
        For any $L\ge 0$, $t\ge k\ge 2$, and $\delta>0$, there exists a constant $c> 0$ such that the following holds. For any union~$\mathcal{F}$ of $n^{1-\delta}$ many $\frac{1}{k-1}$-dense sequences of $k$-graphs and positive integer~$n$, there exists an $n$-vertex $k$-graph~$H$ and an edge coloring with $Ln$ colors satisfying the following conditions:
    \begin{itemize}
        \item each color class is~$K_t^{(k)}$,
        \item and the minimum size of~$F\in \mathcal{F}$ such that there exists a rainbow copy of~$F$ in~$H$ is at least $c\log n$.
    \end{itemize} 
\end{theorem}

\section{Proofs and examples}

\subsection{Preliminaries}

For a graph~$H$, the \emph{excess} of~$H$ is defined as the difference between its number of edges and its number of vertices, i.e., $|E(H)|-|V(H)|$.

Bollob\'as and Szemer\'edi~\cite{BS2002} proved the following upper bound on the girth of a graph in terms of its excess.
\begin{theorem}\label{thm:BS}
    For all $n\ge 4$ and $k\ge 2$, every $n$-vertex graph with excess at least~$k$ has girth at most
    \begin{equation}\label{eq:ubongirth}
        \frac{2(n+k)}{3k}(\log_2k+\log_2\log_2k+4). 
    \end{equation}
\end{theorem}
Thus if an $n$-vertex graph has excess $\Omega(n)$, then its girth is $O(\log n)$.

Two probabilistic tools we will use are Chernoff's bound and Markov's inequality.
\begin{theorem}[Chernoff]\label{thm:chernoff}
    Let $X$ be a binomial random variable $Bin(n,p)$. For any $\epsilon\in (0,1)$, 
    \[\cP(X\le (1-\epsilon)\cE X ) \le e^{-\epsilon^2\cE X/3}. \]
\end{theorem}

\begin{theorem}[Markov]\label{thm:markov}
    For any non-negative random variable~$X$ and $t>0$,
    \[\cP(X\ge t) \le \frac{\cE X}{t}.  \]
\end{theorem}

\subsection{Proof of~\refT{thm:main}}
As mentioned, we prove a slightly stronger result compared to~\refT{thm:main}.
\begin{theorem}\label{thm:mainstronger}
For any $\alpha, \beta\ge 0$ with $2\alpha+\beta>1$, there exist $\xi(\alpha,\beta)>0$ and $C(\alpha,\beta)>0$ such that the following holds. Let~$G$ be an $n$-vertex graph with an edge coloring. If the family of color classes $\cF=(F_1,\dots, F_m)$ satisfies the following conditions:
\begin{itemize}
\item each color class is non-empty,
 \item    there exists $\cF_M\subseteq \cF$ such that $|\cF_M|\ge (\alpha-\xi) n$ and each color class in~$\cF_M$ contains a matching of size two,
 \item there exists $\cF_S\subseteq \cF\setminus\cF_M$ such that $|\cF_S|\ge(\beta-\xi) n$ and each color class in~$\cF_S$ contains a star of size two,
 \item and $|\cF\setminus(\cF_M\cup\cF_S)|\ge (1-\alpha-\beta-\xi)n$.
\end{itemize}
Then the rainbow girth of~$G$ is at most $C\log n$.
\end{theorem}

\begin{proof}
    Without loss of generality, we may assume that $\max(\alpha,\beta)\le 1$, otherwise for $\xi=\frac{\max(\alpha,\beta)-1}{2}$,  by~\refT{thm:BS} taking one arbitrary edge from each color class already yields logarithmic rainbow girth.

    For some~$p:=1-t\in (0,1)$ to be determined later, let~$S$ be a $p$-random subset of~$V(G)=[n]$, i.e., each vertex of~$G$ is included in~$S$ independently with probability~$p$.
    We construct a subgraph~$H$ of~$G$ in the following way: the vertex set of~$H$ is~$S$, and for each color class of~$G$, we include in~$H$ one arbitrary edge from that class, provided the edge is entirely contained within~$S$. Note that~$H$ constructed in this way is rainbow. Let 
    \[K=|E(H)|-|V(H)|\] 
    be the excess of~$H$. We aim to show that
    \begin{equation}\label{eq:lbE}
        \cE K\ge \delta n
    \end{equation}
    for some $\delta(\alpha, \beta)>0$. This implies that there exists an instance of~$H$ whose excess is at least~$\delta n$, and then by the argument below~\refT{thm:BS}, there exists a cycle in~$H$, which is rainbow in~$G$, of length at most $C\log n$ for some $C(\delta)>0$. This will complete the proof.

To prove~\eqref{eq:lbE}, for a matching of size two in~$G$, by inclusion-exclusion, the probability that one of its edges is contained in~$S$ is $2p^2-p^4$. 
For a star of size two, the probability that one of its edges is contained in~$S$ is $2p^2-p^3=p(1-(1-p)^2)$.
And for a single edge, the probability that it is contained in~$S$ is $p^2$.

Combining these probabilities with the conditions that $|\mathcal{F}_M|\ge (\alpha-\xi)n$, $|\mathcal{F}_S|\ge (\beta -\xi)n$, and $|\mathcal{F}\setminus (\mathcal{F}_M\cup\mathcal{F}_S)|\ge (1-\alpha -\beta -\xi)n$, by linearity of expectation we have
\begin{align*}
    \cE K =& \cE |E(H)| -\cE |V(H)|\\
    \ge& (\alpha-\xi) n (2p^2-p^4)+(\beta-\xi) n(2p^2-p^3)+(1-\alpha-\beta-\xi)n p^2-np.
\end{align*}
Substituting $p=1-t$, it implies that
\begin{align*}
    \cE K \ge & (2\alpha +\beta -1)tn +(-\alpha t^2+4\alpha t -5\alpha +\beta t-2\beta +1 )t^2n\\
     &+ (-3+3 t+4 t^2-5 t^3+ t^4)\xi n.
\end{align*}
Setting $\gamma:=2\alpha+\beta-1>0$, $t:=\frac{\gamma}{40}\le 1$, and $\xi:=\frac{\gamma t}{100}$, it yields
\begin{align*}
    \cE K \ge \gamma tn - 20\cdot t (tn) - 20 \xi n =\gamma t n -\frac{1}{2}\gamma t n -\frac{1}{5}\gamma tn\ge \frac{1}{10}\gamma t n.
\end{align*}
Setting $\delta :=\frac{1}{10}\gamma t$ completes the proof of~\eqref{eq:lbE}.
\end{proof}
\begin{remark}
    The condition~$2\alpha+\beta>1$ in~\refT{thm:mainstronger} (or~\refT{thm:main}) is equivalent to~$\alpha>1-\alpha-\beta$, i.e., the proportion of classes that contain a matching of size two is greater than the proportion consisting of a single edge.
\end{remark}

\subsection{Proof of~\refT{thm:nonstar}}
\begin{proof}[Proof of~\refT{thm:nonstar}]
    If there are at least $\frac{\alpha}{2}n$ color classes of~$G$, each of which contains a matching of size two, then~\refC{cor:matching} implies that the rainbow girth is at most $C_1\log n$ for some constant~$C_1>0$ depending on~$\alpha$. 

Otherwise, there are at least $\frac{\alpha}{2}n$ color classes of~$G$, each of which contains a triangle. We take arbitrarily two edges from a triangle in each of such color classes, and arbitrarily one edge from each of the remaining color classes. Then we get a graph~$F$ which has at least 
\[  \frac{\alpha}{2}n\cdot 2 +(n-\frac{\alpha}{2}n)=(1+\frac{\alpha}{2})n \]
edges on~$n$ vertices. As the excess of~$F$ is at least~$\frac{\alpha}{2}n$,~\refT{thm:BS} implies that there exists a constant $C_2(\alpha)>0$ such that the girth of~$F$ is at most $C_2\log n$. If the shortest cycle in~$F$ is not rainbow, two edges of the same color must come from a monochromatic triangle by the construction of~$F$, then we can replace these two edges by the third edge of the triangle and obtain a shorter cycle in~$G$. Do this replacement repeatedly until we get a rainbow cycle, which is a rainbow cycle in~$G$ of length at most $C_2\log n$.

Taking $C=\max(C_1,C_2)$ completes the proof.
\end{proof}

The following result provides a more explicit upper bound on the rainbow girth in terms of the number of non-star color classes.

\begin{theorem}\label{thm:nonstarex}
There exists $L_0>0$ such that for any $0\le c\le \frac{1}{2}$ and $L\ge L_0$, the following holds. Let~$G$ be an $n$-vertex graph admitting an edge coloring with~$n$ colors such that each color class is of size at least two. If there are at least~$Ln^{1-c}$ color classes, each of which contains either a matching of size two or a triangle, then the rainbow girth of~$G$ is at most $\frac{2\log_2 (\frac{L}{10^2}n^{1-2c})}{\frac{L}{10^2}n^{1-2c}}n$.
\end{theorem}
\begin{proof}
First note that there exists a constant $A>0$ such that for all large~$n$ and $A\le k\le n$, the bound in~\eqref{eq:ubongirth} satisfies that
\begin{equation}\label{eq:weaker}
      \frac{2(n+k)}{3k}(\log_2k+\log_2\log_2k+4) \le \frac{2\log_2 k}{k}n  
\end{equation}
and the function $f(k):=\frac{\log_2 k}{k}$ is decreasing for $k\in[A,+\infty)$.

With some forecast, we set \begin{equation}\label{eq:def:L0}
    L_0:=\max(100A, \; 10^3 ).
\end{equation}

If at least $\frac{L}{2}n^{1-c}$ color classes of~$G$ contains a triangle, following the same way as in the proof of~\refT{thm:nonstar} by taking arbitrary two edges from a triangle in each of such color classes and taking one edge from each of the remaining color classes, we obtain a subgraph~$F$ with excess at least~$\frac{L}{2}n^{1-c}$, which by~\refT{thm:BS} and~\eqref{eq:weaker} yields a rainbow cycle in~$G$ of length at most
\[  \frac{2 \log_2 (\frac{L}{2}n^{1-c})}{\frac{L}{2}n^{1-c}}n\le \frac{2\log_2(\frac{L}{10^2}n^{1-2c})}{\frac{L}{10^2}n^{1-2c}},   \]
where the last inequality is by the monotonicity of $f(k)$.

Otherwise, there are~$T$ color classes of~$G$, each of which contains a matching of size two, for some $T\ge \frac{L}{2}n^{1-c}$. Then each of the other~$n-T$ color classes contains a star of size two. Similarly as the proof of~\refT{thm:nonstar}, for some $p:=1-t\in (0,1)$ to be determined later, we take a $p$-random subset~$S$ of~$V(G)$. And for each color class, we arbitrarily take one edge contained entirely in~$S$, if such an edge exists, to form a rainbow subgraph~$H$ with vertex set~$S$. For the excess~$K$ of~$H$, we have
\begin{align*}
    \cE K &= \cE |E(H)|-\cE |V(H)| \\
    &\ge T(2p^2-p^4)+(n-T)(2p^2-p^3)-np\\
    &=\Big(  Tt - 3 T t^2  + 3 T t^3  - T t^4 \Big) +\Big( - n t^2 +n t^3\Big).
\end{align*}
Setting $t:=\frac{1}{10}n^{-c}\le \frac{1}{10}$ and using $T\ge \frac{L}{2}n^{1-c}$, it implies that
\[\cE K \ge \frac{1}{2}tT-nt^2\ge \frac{1}{2}\cdot\frac{1}{10}n^{-c}\cdot \frac{L}{2}n^{1-c}-\frac{1}{10^2}n^{1-2c}\ge \frac{L}{10^2}n^{1-2c}. \]
Therefore there exists an instance~$H$ whose excess is at least $\frac{L}{10^2}n^{1-2c}$. Then by~\eqref{eq:def:L0}, \refT{thm:BS}, and~\eqref{eq:weaker}, it implies that there is a cycle in~$H$ of length at most $\frac{2\log_2 (\frac{L}{10^2}n^{1-2c})}{\frac{L}{10^2}n^{1-2c}}n$, which is rainbow in~$G$.
This completes the proof.
\end{proof}

\subsection{Proof of~\refT{thm:lb}}
For a $t$-graph~$G$ on~$[n]$, the \emph{$k$-shadow} of~$G$ is
\[ \partial_k(G):=\{ S\in \binom{[n]}{k}\mid S\subseteq e \text{ for some $e\in G$} \}.  \]
\begin{proof}[Proof of~\refT{thm:lb}]
For
 \begin{equation*}
        p:= \frac{4\cdot 2^t t! L}{n^{t-1}},
    \end{equation*}
    let~$G_0$ be $G^{(t)}(n,p)$, i.e., each edge of the complete $t$-graph~$K_{n}^{(t)}$ is included in~$G_0$ independently with probability~$p$.
   
The idea to prove the theorem is that by alteration, we shall find some $G\subseteq G_0$ such that $|G|\ge L n$, $|e\cap f|\le 1$ for any distinct $e,f\in G$, and for $H=\partial_k G$ with color classes $( \binom{e}{k} )_{e\in G}$, the minimum size of~$F\in\mathcal{F}$ such that there is a rainbow copy of~$F$ in~$H$ is at least~$c\log n$. Note that by construction, $H=\cup_{e\in G}\binom{e}{k}$, and the condition~$|e\cap f|\le 1<k$ for distinct $e,f\in G$ guarantees that the color classes are disjoint.

    Turn to the details. We assume that~$n$ is large enough in the following. Since 
    \[\cE |G_0|=\binom{n}{t}p \ge \frac{(n-t)^t}{t!}p\ge \frac{(n/2)^t}{t!}p\ge 4Ln,\]
    by Chernoff's bound~\refT{thm:chernoff}, the event
    \[ \mathcal{A}:=\{ |G_0|\ge 3Ln \} \]
    holds with probability at least $\cP(|G_0|\ge 0.9\cE|G_0|)=1-o(1)$.

    Let~$Y$ be the number of pairs $(e,f)$ such that $e,f\in G_0$ are distinct and $|e\cap f|\ge 2$. Since there are at most~$\binom{n}{t}$ ways to choose a $t$-set~$e$, at most $\binom{t}{2}$ ways to determine two vertices in the intersection, and at most~$\binom{n-2}{t-2}$ ways to extend the two vertices to a~$t$-vertex set~$f$, we have
    \[ \cE Y \le \binom{n}{t}\binom{t}{2}\binom{n-2}{t-2}p^2=o(n). \]
    Thus Markov's inequality~\refT{thm:markov} implies that the event
    \[ \mathcal{B}:=\{ Y\le Ln \} \]
holds with probability $1-o(1)$.

For a copy~$C=\{S_1,\dots,S_\ell\}$ of a $k$-graph~$F$ of size~$\ell$ on~$[n]$ and a $t$-graph~$T$ on~$[n]$,~$C$ is called \emph{distinguishable in~$T$} if there exists $e_1,\dots, e_\ell\in T$ such that $S_i\subseteq e_i$ for each $1\le i\le \ell$, and $S_j\not\subseteq e_i$ for $j\neq i$. To bound the probability of~$C$ being distinguishable in~$G_0$, for each $1\le i\le \ell$ let $\mathcal{T}_{S_i}$ be the event that there exists $e_i\in G_0$ such that $S_i\subseteq e_i$ and $S_j\not\subseteq e_i$ for all $j<i$. Then
\begin{align*}
    \cP(\text{$C$ is distinguishable in~$G_0$})=\cP(\cap_{i=1}^\ell \mathcal{T}_{S_i})&=\prod_{i=1}^\ell \cP(\mathcal{T}_{S_i}\mid\cap_{j=1}^{i-1}\mathcal{T}_{S_j}).
\end{align*}
Since there are at most $\binom{n-k}{t-k}$ many~$t$-sets~$e$ satisfying $S_i\subseteq e$ and $S_j\not\subseteq e$ for all~$j<i$,
\[\cP(\neg\mathcal{T}_{S_i}\mid\cap_{j=1}^{i-1}\mathcal{T}_{S_j}) \ge (1-p)^{\binom{n-k}{t-k}}\ge 1-p\binom{n-k}{t-k},    \]
where the last inequality is by Bernoulli's inequality $(1-p)^N\ge 1-pN$ for $p\in(0,1)$ and positive integer~$N$.
Therefore 
\[\cP(\text{$C$ is distinguishable in~$G_0$})\le \Big( p\binom{n-k}{t-k}  \Big)^\ell\le p^\ell n^{\ell(t-k)}.  \]
Let~$X_\ell$ be the number of distinguishable copies on~$[n]$ of some~$k$-graph in~$\mathcal{F}$ of size~$\ell$. Given $F\in \mathcal{F}$ of size~$\ell$, by the density assumption of~$F$, we have $|V(F)|\le \ell(k-1)$, so the number of copies of~$F$ on~$[n]$ is at most~$n^{\ell(k-1)}$. Because there are at most~$n^{1-\delta}$ many~$k$-graphs of size~$\ell$ in~$\mathcal{F}$,
\[ \cE X_\ell \le n^{1-\delta}\cdot n^{\ell(k-1)} p^\ell n^{\ell(t-k)}=n^{1-\delta}(4\cdot 2^t t!L  )^\ell. \]
By choosing $c(L,t,\delta)>0$ small enough (say $c\log (4\cdot 2^t t!L) \le \delta/9$), for $2\le \ell\le c\log n$, we have
\[ \cE X_\ell =o(n^{1-\frac{2\delta}{3}}).  \]
Thus by Markov's inequality, 
\[ \cP(X_\ell\ge n^{1-\frac{\delta}{3}})=o(n^{-\frac{\delta}{3}}). \]
Taking a union bound over all $2\le \ell\le c\log n$, the event
\[\mathcal{C}:=\{ \sum_{2\le \ell\le c\log n}X_\ell \le Ln \} \]
holds with probability $1-o(1)$.
Summing up, we have
\[\cP(\mathcal{A}\cap\mathcal{B}\cap\mathcal{C})=1-o(1).\]
Take an instance~$G_0$ such that the event~$\mathcal{A}\cap\mathcal{B}\cap\mathcal{C}$ holds. The event~$\mathcal{A}$ ensures that $|G_0|\ge 3Ln$. If there two edges of~$G_0$ intersecting at more than one vertex, we remove one of the edges. Then event~$\mathcal{B}$ guarantees that we remove at most~$Ln$ edges. Let the remaining $t$-graph be~$G_1$. If for some $2\le \ell\le c\log n$, there is a copy $\{S_1,\dots, S_\ell\}$ on~$[n]$ of some~$F\in\mathcal{F}$ of size~$\ell$  and $e_1,\dots, e_\ell\in G_1$ such that $S_i\subseteq e_i$ for each $1\le i\le \ell$ and $S_j\not\subseteq e_i$ for $j\neq i$, we remove arbitrary one edge from $e_1,\dots, e_\ell$. Then event~$\mathcal{C}$ guarantees that in this step we remove at most~$Ln$ edges from~$G_1$. Let the resulting $t$-graph be~$G$ and let~$H:=\partial_k(G)$. Since $G\subseteq G_1$, $|e\cap f|\le 1$ for any distinct $e,f\in G$. And there is no distinguishable copy of~$F$ in~$G$ for $F\in\mathcal{F}$ with $|F| \le c\log n$. 

Then~$H$ satisfies the conclusion of the theorem: let the color classes of~$H$ be~$(\binom{e}{k})_{e\in G}$, so there are $|G|\ge |G_0|-Ln-Ln\ge Ln $ many color classes and each class is a copy of~$K^{(k)}_t$. As discussed, the classes are disjoint and their union is~$H$. Furthermore, if there is a rainbow copy of~$F$ in~$H$ for some $F\in\mathcal{F}$, then the copy is distinguishable in~$G$, which implies that~$|F| >c \log n$. 
\end{proof}

\subsection{Tightness of the condition $2\alpha+\beta >1$ in~\refT{thm:main}}
To get logarithmic rainbow girth in~\refT{thm:main}, it is necessary to assume that $2\alpha+\beta >1$.  For $\alpha,\beta\ge 0$ with $2\alpha+\beta=1$, the following is an example of an $n$-vertex graph satisfying all the other conditions in~\refT{thm:main}, whose rainbow girth is linear in~$n$.
\begin{example}\label{ex:abtight}
When $\alpha =0$ or~$\beta =0$, the tightness has been shown by the example below~\refCon{conj:A}, or the example mentioned below~\refT{thm:matchings}, respectively.

    For $\min(\alpha,\beta)>0$ with $2\alpha +\beta =1$, 
    we assume that the integer~$\alpha n$ is at least four and is divisible by two. We shall construct an $n$-vertex~$G$ and an edge coloring with~$n$ colors, such that $\alpha n$ color classes are matchings of size two, $\beta n$ color classes are stars of size two, and $(1-\alpha -\beta )n=\alpha n$ color classes are of size one. Furthermore the rainbow girth of~$G$ is  at least $\min( \alpha n, \frac{\beta}{2}n )$, which is linear in~$n$.

The graph has two connected components on $X=\{x_1,\dots, x_{2\alpha n} \}$ and $Y=\{y_1,\dots, y_{\beta n} \}$, respectively. The~$\alpha n$ many edge sets $\{x_{4i+1}x_{4i+2}, x_{4i+3}x_{4i+4}\}$ and $\{x_{4i+2}x_{4i+5},x_{4i+4}x_{4i+7} \}$ for $i=1,\dots, \frac{\alpha n}{2}$ are color classes that are matchings of size two, where the indices are taken modulo $2\alpha n$. The~$\alpha n$ many edge sets $\{x_{4i+1}x_{4i+3} \}$ and $\{x_{4i+2}x_{4i+4} \}$  for $i=1,\dots, \frac{\alpha n}{2}$ are color classes consisting of a single edge, where the indices are taken modulo $2\alpha n$. The $\beta n$ many edge sets $\{y_iy_{i+1},y_iy_{i+2} \}$ are color classes that are stars of size two, where the indices are taken modulo $\beta n$.
Then a rainbow cycle of~$G$ is either in~$X$ or in~$Y$, which has length at least  $\min( \alpha n, \frac{\beta}{2}n )$.
\end{example}
\small
\bibliographystyle{abbrv}

\normalsize

\end{document}